\documentclass[article]{elsarticle}

\usepackage{hyperref}
\usepackage{geometry}
\usepackage{amsmath,amssymb,enumerate,amsfonts}

\usepackage[normalem]{ulem} 
\usepackage{longtable,threeparttable,float}
\usepackage[table]{xcolor}
\usepackage{graphicx}
\usepackage{caption}
\usepackage{listings}%
\usepackage{subcaption}
\usepackage{algorithm}%
\usepackage{algorithmicx}%
\usepackage{algpseudocode}%
\usepackage{epstopdf}
\usepackage{tabularx}
\usepackage[justification=centering]{caption}
\usepackage {color}
\usepackage{hyperref} 
\hypersetup{
	colorlinks=true,
	citecolor=blue,
	linkcolor=blue,
	urlcolor=green}

\usepackage{setspace}
{\begin{list}{}{\settowidth{\labelwidth}{\bf #1}%
			\setlength{\leftmargin}{\labelwidth}%
			\addtolength{\leftmargin}{\labelsep}%
			}}%
	{\end{list}}

\usepackage[figuresright]{rotating}

\journal{Journal of \LaTeX\ Templates}

\newtheorem{definition}{Definition}[section]
\newtheorem{remark}{Remark}[section]

\newtheorem{theorem}{Theorem}[section]
\newtheorem{proposition}{Proposition}[section]
\newtheorem{lemma}{Lemma}[section]
\newtheorem{corollary}{Corollary}[section]
\newtheorem{assumption}{Assumption}[section]

\newenvironment{proof}{{\noindent \it Proof.}}{\hfill $\square$\par}

\begin{document}

\begin{frontmatter}

\title{Decomposition-Invariant Pairwise Frank-Wolfe Algorithm for Constrained Multiobjective Optimization}

\author[mymainaddress]{Zhuoxin Fan}
\ead{zxfan42@gmail.com}

\author[mymainaddress]{Liping Tang}
\ead{tanglps@163.com}

\address[mymainaddress]{National Center for Applied Mathematics in Chongqing, Chongqing Normal University, Chongqing, 401331, China}

\begin{abstract}
Recently the away-step Frank-Wolfe algoritm for constrained multiobjective optimization has been shown linear convergence rate over a polytope which is generated by finite points set. In this paper we design a decomposition-invariant pairwise frank-wolfe algorithm for multiobjective optimization that the feasible region is an arbitrary bounded polytope. We prove it has linear convergence rate of the whole sequence to a pareto optimal solution under strongly convexity without other assumptions.
\end{abstract}

\begin{keyword}
Multiobjective optimization \sep Frank-Wolfe method \sep Linear convergence rate \sep Decomposition-invariant
\end{keyword}

\end{frontmatter}

\section{Introduction}\label{sec:1}

Constrained multiobjective optimization problem has casued much concern. Many problems in practical applications can be expressed as it, such as in engineering design\cite{eschenauer2012multicriteria}, statics\cite{carrizosa1998dominating}, machine learning\cite{jin2007multi}.

Recently, The gradient-type method for solving the multiobjectvie optimization problem is popular. The first is steepest descent method\cite{fliege2000steepest} which is solving unconstrained multiobjectvie optimization problem. After that, projected gradient method\cite{drummond2004projected}, proximal point method\cite{tanabe2019proximal} and conditional gradient method\cite{assunccao2021conditional} have developed for solving constrained multiobjective optimization.

The conditional gradient method also named Frank-Wolfe method was first proposed by \cite{frank1956algorithm} for solving convex continuously differentiable optimization problems. And it has been extended to single-objective\cite{levitin1966constrained} and multiobjectvie optimization\cite{assunccao2021conditional}. Frank-Wolfe method solves a linear optimization problem in each iteration, so it has advantages of simplicity and ease of implementation compared to projected gradient method. But the disadvantange of Frank-Wolfe method is obvious: its convergence rate is $\mathcal{O}(1/k)$ and can not be improved for general constrained convex optimization problems\cite{jaggi2013revisiting}\cite{lan2013complexity}. So many improvements to this disadvantange have been proposed such as pairwise, away-step Frank-Wolfe method\cite{guelat1986some}\cite{lacoste2015global} and conditional gradient sliding algorithm\cite{lan2016conditional}. For constrained multiobjective optimization, \cite{gonccalves2024away} proposed the away-step Frank-Wolfe method. It attains linear convergence rate under certain assumptions. But the algorithm is to solve the problem that feasible region is a polytope generated by finite set of points. In this paper, we consider the feasible region is arbitrary polytopes. \cite{bashiri2017decomposition} and \cite{garber2016linear} have considered the single-objective optimization problems that these feasible regions are aribitary oplytopes. Combining with the ideas of them, we propose a new subproblem to calculate descent direction which we call pairwise direction. Under adaptive step size which is the natrual extension to multiobjective optimization, our algorithm also attain linear convergence rate when the objective functions have strongly convexity. Compared to the assumptions in away-step Frank-Wolfe algorithm for constrained multiobjective optimization, our assumptions do not need the additional assumption inequality (27) in \cite{gonccalves2024away} and is same as in \cite{lacoste2015global}.

The paper is organized as follows: Sect.\ref{sec:2} introduces some notations, definitions and lemmas we will use. Sect.\ref{sec:3} contains our main result: the decomposition-invariant pairwise Frank-Wolfe algorithm for multiobjective optimization. In Sect.\ref{sec:4} we proved asymptotic convergence and the algorithm has linear convergence rate when objective functions are strongly convex. Sect.\ref{sec:5} gives numberical experiments. Sect.\ref{sec:6} contains the conclusion of this paper.

\section{Preliminaries}\label{sec:2}
This section, we give some notations and definitions. If not otherwise specified, the norm of this paper is Eulidean norm. Let $[n]$ denotes $\{1,2,...,n\}$, $[m]:=\{1,...,m\},R^n_+:=\{u\in R^n:u_j\geq 0,j\in [m]\}$, and $R^n_{++}:=\{u\in R^n:u_j>0,j\in [m]\}$. For $u,v\in R^n,v\succeq u$ (or $u\preceq v$) means that $v-u\in R^n_+$ and $v\succ u$ (or $u\prec v$) means that $v-u\in R^n_{++}$. In general, a polytope $\mathcal{P}$ can be defined as
\begin{equation*}
	\mathcal{P}=\{x\in R^n:Ax\leq b,Cx=d\},
\end{equation*}
where $A=(a_1,...,a_{m_1})^\top\in R^{m_1\times n}$, $b=(b_1,...,b_{m_1})^\top\in R^{m_1}$, $C=(c_1,...,c_{m_2})^\top\in R^{m_2\times n}$ and $d=(d_1,...,d_{m_2})^\top\in R^{m_2}$. We denote the set of vertices of $\mathcal{P}$ is $\mathcal{V}$, the diameter of $\mathcal{P}$ is $D=\sup\{\Vert x-y\Vert:x,y\in\mathcal{P}\}$. In this paper we only concider $D<\infty$. 
We denote this constrained problem as 
\begin{eqnarray}
	\min &F(x)=(f_1(x),...,f_m(x))^\top\nonumber\\
	s.t. &Ax\leq b,\label{pro:1}\\
	&Cx=d.\nonumber
\end{eqnarray}
Throughout the paper we assume that $F: R^n\to R^m$ is a countinuously differentiable function given by $F :=(f_1,...,f_m)^T.$
\begin{definition}\normalfont
	A function $F:=(f_1,...,f_m)^T$ is $Lipschitz\ continuous$ with respect to $\vert\vert\cdot\vert\vert$, if there exist constants $L_1,...,L_m>0$ such that
	\begin{equation}
		f_j(y)\leq f_j(x)+\nabla f_j(x)\cdot (y-x)+\frac{L_j}{2}\vert\vert x-y\vert\vert^2,\ \forall x,y\in \mathcal{C},\ \forall j\in [m].
	\end{equation}
	For future reference we set $L:=\max\{L_j:j\in [m]\}.$
\end{definition}

\begin{definition}\normalfont
	A function $F(x)=(f_1(x),...,f_m(x))^\top$ is convex over a convex set $\mathcal{C}\subset R^n$ if for all $j\in [m]$ it satisfies the following equivalent condtions
	\begin{eqnarray}
		f_j(y)\leq f_j(x)+\langle\nabla f_j(x),y-x\rangle\ \forall x,y\in\mathcal{C}.
	\end{eqnarray}
\end{definition}
A point $x^*\in \mathcal{P}$ is called weak pareto optimal point for Problem (\ref{pro:1}), if there exists no other $x\in \mathcal{P}$ such that $F(x)\prec F(x^*)$. $\mathcal{W}^*$ denotes the set of weak pareto potimal point for Problem (\ref{pro:1}). For a face $\mathcal{F}$ of $\mathcal{P}$, define $\dim \mathcal{F}$ as:
\begin{equation*}
	\dim \mathcal{F}:=n-\dim \mathtt{span}\{\{a_1,...,a_{m_1}\}\cup \{c_i:1\leq i\leq m_2,c_i^\top x=d_i,\forall x\in \mathcal{F}\}\}.
\end{equation*}
Let $\mathcal{F}^*\subseteq \mathcal{P}$ is the lowest-dimensional face of $\mathcal{P}$ such that $\mathcal{W}^*\subseteq \mathcal{F}^*$. $\mathcal{F}^*$ can be written as $\mathcal{F}^*=\{x\in R^n:A^* x\leq b^*,C^* x=d^*\}$. The rows of $A^*$ are the rows of $A$ that correspond to inequality constraints that are satisfied by some of the points in $\mathcal{F}^*$ but not by others, and $b^*$ is defined accordingly. The rows of $C^*$ are the rows of $C$ plus the rows of $A$ that correspond to inequality constraints that are tight for all points in $\mathcal{F}^*$ and $d^*$ is defiend accordingly. 
Denote $LI(\mathcal{P})$ the set of all $\dim \mathcal{F}^*\times n$ matrices whose rows are linearly independent rows chosen from the rows of $C^*$. We define 
\begin{equation*}
	\phi^*=\max_{M\in LI(\mathcal{P})}\Vert M\Vert,\xi^*=\min_{v\in \mathcal{V}\cap \mathcal{F}^*}\min_i\{b(i)-A(i)^\top v|b(i)>A(i)^\top v\}.
\end{equation*}
$\phi^*$ and $\xi^*$ are constants for a polytope. 
For simplity, we define a useful function as
\begin{equation}
	\sigma(x,z):=\min_{j\in [m]}[f_j(x)-f_j(z)].
\end{equation}

\section{Algorithms}\label{sec:3}
The subproblem of classical frank-wolfe algorithm for multiobjective optimization\cite{assunccao2021conditional} in the k-th iteration which we named FW-subproblem is
\begin{equation}
	\min_{u\in \mathcal{P}}\max_{j\in [m]}\langle \nabla f_j(x^k),u-x^k\rangle.\label{fwsub}
\end{equation}
We denote the optimal solution and value of it by $p^{FW}(x^k)$ and $\theta^{FW}(x^k)$, which means
\begin{equation}
	p^{FW}(x^k)\in \arg\min_{u\in \mathcal{P}}\max_{j\in [m]}\langle \nabla f_j(x^k),u-x^k\rangle,
\end{equation}
and
\begin{equation}
	\theta^{FW}(x^k)= \min_{u\in \mathcal{P}}\max_{j\in [m]}\langle \nabla f_j(x^k),u-x^k\rangle.\label{difwsub}
\end{equation}
The following lemma which is in \cite{assunccao2021conditional} describle the property of $\theta^{FW}(x)$.
\begin{lemma}
	For $\theta^{FW}(x)$, $x\in \mathcal{P}$, the following conclusions hold:
	\begin{itemize}
		\item[(i)] $\theta^{FW}(x)\leq 0$;
		\item[(ii)] $\theta^{FW}(x)$ is continous;
		\item[(iii)] $\theta^{FW}(x^*)=0$ if and only if $x^*\in \mathcal{P}$ is a weak Pareto optimal point of Problem (\ref{pro:1}).
	\end{itemize}
	\label{lemma:1}
\end{lemma}
The idea of away-step frank-wolfe algorithm \cite{guelat1986some}\cite{lacoste2015global} is to compute the classical diection and the away-step direction, then choose the better one. Our direction adopts the decompostion-invariant approach simillar with \cite{bashiri2017decomposition} and \cite{garber2016linear}. Our search space is restricted to the vertices that satisfy all the inequality constraints by equality if current $x^k$ does so:
\begin{eqnarray}
	& &\min_{u,v\in R^n}\max_{j\in [m]}\langle \nabla f_j(x^k),u-v\rangle,\nonumber\\
	&s.t.&Au\leq b,Cu=d\\
	& &Av\leq b,Cv=d,\ and\ \langle a_i,x^k\rangle=b_i\to \langle a_i,u\rangle=b_i,\forall i.\nonumber
\end{eqnarray}
The optimal value denotes $\theta^{PW}(x^k)$, the optimal solution denotes $(p^{PW}(x^k),q^{PW}(x^k))$. This question is the has the following proposition. Its proof is the same as property 1 in \cite{bashiri2017decomposition}. In fact, property 1 in \cite{bashiri2017decomposition} is a special situation.
\begin{proposition}
	Denote $\mathcal{S}(x):=\{S\subset \mathcal{P}:x =\sum^t_{i=1}\lambda_i s_i,s_i \in S\ and\ \lambda_i>0\}$. Then (8) is equivalent to
	\begin{equation}
		\min_{S\in \mathcal{S}(x)}\min_{u\in R^n,v\in S}\max_{j\in [m]}\langle \nabla f_j(x),u-v\rangle\ 
		s.t.Au\leq b,Cu=d.
	\end{equation}
	\label{prop:1}
\end{proposition}
We can get the relationship between $\theta^{PW}(x)$ and $\theta^{FW}(x)$ easily from Proposition (\ref{prop:1}).
\begin{corollary}
	For any $x$ which is feasible, we have $\theta^{PW}(x)\leq \theta^{FW}(x)$.
\end{corollary}

\begin{algorithm}[h]
	\caption{Decomposition-Invariant Pisewise Frank-Wolfe Algorithm for Multiobjective Optimization}
	\begin{algorithmic}[line]
		\State Initialization: $x_1\in \mathcal{P},\epsilon>0$.
		\For {$k=1,2,...$}
		\State calculate pisewise direction $d^{PW}_k$ as ().
		\If {$\|\theta^{FW}(x^k)\|\leq \epsilon$}
		\State Return $x^k$.
		\EndIf
		\State calculate $\lambda_{k,\max}$ such that $\lambda_{k,\max}=\max\{\lambda\geq 0:x^k+\lambda_k d^{PW}_k\in \mathcal{P}\} $.
		\State Choose the step size $\lambda_k \in[0,\lambda_{k,\max}]$.\EndFor
	\end{algorithmic}
	\label{Algo:1}
\end{algorithm}
In Algorithm \ref{Algo:1}, the choice of step size is different from classical Frank-Wolfe algorithm. It is beacuse the algorithm should ensure that $x^{k+1}$ remains in the feasible region. No general solution is known to this difficulty. Our approach is doing line search or adaptive step size $\lambda_k =\min\{\lambda_{k,\max},\frac{-\theta^{PW}(x^k)}{L\Vert d^{PW}(x^k)\Vert}\}$.
\begin{algorithm}
	\caption{Adaptive Step Size}
	\begin{algorithmic}[line]
		\State Input:$F(x)$, $x\in \mathcal{C}$ and descent direction $d(x)$, smoothness parameter $L>0$,step size parameter $\lambda_{max}$ and progress parameters $\eta<1<\tau$.
		\State $M=\eta L,\lambda=\min\{\lambda_{\max},\frac{-\max_{j\in [m]}\langle \nabla f_j(x),d(x)\rangle}{M\Vert d(x)\Vert^2}\}$.
		\While{$$\{\max_{j\in [m]}(f_j(x+\lambda d(x))-f_j(x))
			\geq \lambda\max_{j\in [m]}\langle \nabla f_j(x),d(x)\rangle+\frac{\lambda^2 M}{2}\Vert d(x)\Vert^2\} $$}
		
		\State $M=\tau M,\lambda=\min\{\lambda_{\max},\frac{-\max_{j\in [m]}\langle \nabla f_j(x),d(x)\rangle}{M\Vert d(x)\Vert^2}\}$.
		\EndWhile
		\State Ouput: updated smoothness parameter $L=M$ and step size $\lambda$.
	\end{algorithmic}
	\label{Algo:2}
\end{algorithm}
Algorithm \ref{Algo:2} is a generalized form of adaptive step size in \cite{pedregosa2020linearly}. We use 
\begin{equation}
	\max_{j\in [m]}(f_j(x+\lambda d(x))-f_j(x))
	\geq \lambda\max_{j\in [m]}\langle \nabla f_j(x),d(x)\rangle+\frac{\lambda^2 M}{2}\Vert d(x)\Vert^2
\end{equation}
instead of
\begin{equation}
	\{(f_j(x+\lambda d(x))-f_j(x))\geq \lambda\langle \nabla f_j(x),d(x)\rangle+\frac{\lambda^2 M}{2}\Vert d(x)\Vert^2\},\forall j\in [m].
\end{equation}
Obviously, the first is weaker than the last inequality. And Algorithm \ref{Algo:2} using either of these two ensures the following holds
\begin{equation*}
	\max_{j\in [m]}(f_j(x+\lambda d(x))-f_j(x))
	\leq \lambda\max_{j\in [m]}\langle \nabla f_j(x),d(x)\rangle+\frac{\lambda^2 L}{2}\Vert d(x)\Vert^2,
\end{equation*}
where $L,\lambda$ are output of Algorithm \ref{Algo:2}. Hence we choose the first inequality.
\section{Convergence Analysis}\label{sec:4}
In this section, we discuss the convergence analysis of Algorithm \ref{Algo:1}. Firstly, we discuss asymptotic convergence of Algorithm \ref{Algo:1}. Then we prove Algorithm \ref{Algo:1} has $\mathcal{O}(1/k)$ for general constraint multiobjective convex optimization problem. However, if the objective function satisfies \eqref{assu:1} which is weaker than strongly convexity, we can get linear convergence rate.
\subsection{Asymptotic Convergence}
We present the asymptotic convergence analysis of Algorithm \ref{Algo:1} in this subsection. In Theorem \ref{theo:1}, we have proved every limit point of $\{x^k:k\in \mathbb{N}\}$ generated by Algorithm \ref{Algo:1} is a weak pareto point of Problem (\ref{pro:1}). If the objective function is strongly convex, every limit point of $\{x^k:k\in \mathbb{N}\}$ is a pareto point of Problem (\ref{pro:1}). 
\begin{lemma}
	Let $\{x^k:k\in \mathbb{N}\}$ be generated by Algorithm \ref{Algo:1}, we have
	\begin{equation}
		f_j(x^{k+1})\leq f_j(x^k)+\lambda_k\theta^{PW}(x^k)+\frac{L}{2}\lambda_k^2\Vert d^{PW}(x^k)\Vert^2,\forall j\in [m].
	\end{equation}
	\label{lemma:2}
\end{lemma}
\begin{proof}
	Since $F(x)=(f_1(x),...,f_m(x))^\top$ is Lipschitz continuous with $L$ on $\mathcal{P}$, we have
	\begin{eqnarray*}
		f_j(x^{k+1})&\leq& f_j(x^k)+\lambda_k\langle\nabla f_j(x^k),d^{PW}(x^k)\rangle+\frac{L}{2}\lambda_k^2\Vert d^{PW}(x^k)\Vert^2\\
		&\leq&f_j(x^k)+\lambda_k\max_{j\in [m]}\langle\nabla f_j(x^k),d^{PW}(x^k)\rangle+\frac{L}{2}\lambda_k^2\Vert d^{PW}(x^k)\Vert^2\\
		&\leq&f_j(x^k)+\lambda_k\theta^{PW}(x^k)+\frac{L}{2}\lambda_k^2\Vert d^{PW}(x^k)\Vert^2.
	\end{eqnarray*}
\end{proof}
\begin{remark}\normalfont
	From the Lemma \label{lemma:2} and the choice of step size $\lambda_k$, if none of the inequality constraints are satisfied as equality at the optimal $\lambda_k$ of line search, we call it a good step, otherwise we call it a bad step. For a bad step, we can only have that $f_j(x^{k+1})\leq f_j(x^k),\forall j\in [m]$. Fortunately we can bound the number of bad steps between two good steps.\\
	First, if $\lambda_{k,\max}=1$ it is a good step. We know that if $d^{PW}(x^k)=x^k$, $\lambda_{k,\max}=1$ holds. If $\lambda_{k,\max}<1$, at least one inequality constraint will turn into an equality constraint when doing line search. Algorithm \ref{Algo:1} selects the search direction $d^{PW}(x^k)$ by respecting all equality constraints, so succession of the search direction $d^{PW}(x^k)$ will terminate when the set of equalities only define a singleton $x^k$. Since $x\in R^n$, it is obviously that the number of bad steps $bad(\mathcal{P})$ between two good steps is less than $n$.
	\label{remark:1}
\end{remark}
\begin{theorem}
	Every limit point of $\{x^k:k\in \mathbb{N}\}$ generated by Algorithm \ref{Algo:1} is a weak Pareto optimal point of Problem (\ref{pro:1}).
	\label{theo:1}
\end{theorem}
\begin{proof}
	From the Remark \ref{remark:1}, we know there exists a subsequence of $\{x^k:k\in \mathbb{N}\}$ such that every in it is a good step. assume that the indices of subsequence is $K=\{k_0,k_1,...\}\subset \mathbb{N}$. Since $\mathcal{P}$ is compact and $x^k\in \mathcal{P}$ for all $k\in K$, we assume without loss of generality that $\{x^k:k\in K\}$ converge to $\bar{x}\in\mathcal{P}$. From the definition of $\theta^{FW}(x^k)$ and $\theta^{PW}(x^k)$, we have $\theta^{PW}(x^k)\leq \theta^{FW}(x^k)<0$, $\Vert d^{PW}(x^k)\Vert\leq D$ where $D$ that is the diameter of $\mathcal{P}$. From Lemma \label{lemma:2}, we have
	\begin{equation*}
		f_j(x^{k_{i+1}})\leq f_j(x^{k_i})+\lambda_{k_i}\theta^{PW}(x^{k_i})+\frac{L}{2}\lambda_{k_i}^2\Vert d^{PW}(x^{k_i})\Vert^2,\forall j\in [m],k_i\in K.
	\end{equation*}
	Since it is good step when $k_i\in K$, we have
	\begin{equation*}
		f_j(x^{k_{i+1}})\leq f_j(x^{k_i})-\frac{\theta^{PW}(x^{k_i})^2}{2L\Vert d^{PW}(X^{k_i})\Vert^2}, \forall j\in [m].
	\end{equation*}
	It follows that
	\begin{equation*}
		\frac{\theta^{FW}(x^{k_i})^2}{2LD^2}\leq\frac{\theta^{PW}(x^{k_i})^2}{2LD^2}\leq f_j(x^{k_i})-f_j(x^{k_{i+1}}),\forall j\in [m].
	\end{equation*}
	But the sequence $\{F(x^k):k\in \mathbb{N}\}$ is monotonically decreaing, we have that the whole sequence $\{F(x^k):k\in \mathbb{N}\}$ converges to some point in $R^m$. It is because $\{x^k:k\in \mathbb{N}\}\subset \mathcal{P}$, $\mathcal{P}$ is compact and $F(x)$ is Lipschitz continuous.
	So we have that
	\begin{equation*}
		\lim_{i\to \infty}\frac{\theta^{FW}(x^{k_i})^2}{2L\Vert d^{FW}(x^{k_i})\Vert^2}=0.
	\end{equation*}
	From Lemma \ref{lemma:1}, $\theta^{FW}(\bar{x})=0$ means that $\bar{x}$ is a weak Pareto optimal point for Problem (\ref{pro:1}). And $\{F(x^k):k\in \mathbb{N}\}$ converges to $F(\bar{x})$, hence any limit point $x^*\in \mathcal{P}$ of $\{x^k:k\in \mathbb{N}\}$ is a weak Pareto point for Problem (\ref{pro:1}).
\end{proof}
\begin{corollary}
	If $F(x)$ is $\mu$-strongly convex, then $\{x^k:k\in \mathbb{N}\}$ converges to a Pareto optimal point of Problem (\ref{pro:1}).
\end{corollary}
\begin{proof}
	From Theorem \ref{theo:1}, we know that a limit point $\bar{x}$ of $x^*\in \mathcal{P}$ is a weak Pareto point for Problem (\ref{pro:1}), it is also a critial point. So we have
	\begin{equation*}
		\max_{j\in [m]}\langle \nabla f_j(\bar{x}),x^k-\bar{x}\rangle\geq 0, \forall j\in [m].
	\end{equation*}
	Since $F(x)$ is $\mu$-strongly convex, we have
	\begin{equation*}
		f_j(x^{k})-f_j(\bar{x})\geq\langle \nabla f_j(\bar{x}),x^k-\bar{x}\rangle+\frac{\mu}{2}\Vert\bar{x}-x^k\Vert^2,\forall j\in [m],k\in\mathbb{N}.
	\end{equation*}
	It follows that
	\begin{equation*}
		\max_{j\in [m]}\{f_j(x^k)-f_j(\bar{x})\}\geq \max_{j\in [m]}\langle \nabla f_j(\bar{x}),x^k-\bar{x}\rangle+\frac{\mu}{2}\Vert\bar{x}-x^k\Vert^2\geq \frac{\mu}{2}\Vert\bar{x}-x^k\Vert^2.
	\end{equation*}
	From the Theorem \ref{theo:1}, we know that $F(x^k)$ converges to a point that we conclude it as $F^*=(f_1^*,...,f_m^*)^\top$. Hence we have that $\{x^k:k\in\mathbb{N}\}$ converges to $\bar{x}$. Since $F(x)$ is strongly convex and $\bar{x}$ is weak Pareto point for Problem (\ref{pro:1}), it follows that $\bar{x}$ is a Pareto point for Problem (\ref{pro:1}).
\end{proof}
\subsection{Linear Convergence}
We prove Algorithm \ref{Algo:1} has $\mathcal{O}(1/k)$ when the objective function is general convex function. Lemma \label{lemma:5} gives the upper bound of $\theta^{PW}(x^k)$. Then we attains linear convergence rate when the objective function satisfies \eqref{assu:1}. 
\begin{assumption}\normalfont
	Let $\{x^k:k\in \mathbb{N}\}$ generated by Algorithm \ref{Algo:1} converges to a pareto optimal point $x^*$. There exists $\mu>0$ such that for all $j\in [m]$, $f_j(x)$ we have
	\begin{equation}
		\langle \nabla f_j(x^k),x^*-x^k\rangle\leq -\Vert x^k-x^*\Vert \sqrt{2\mu(f_j(x^k)-f_j(x^*))},\forall j\in [m].\label{assu:1}
	\end{equation}
\end{assumption}
\begin{remark}\normalfont
	\label{remark:2}
	When $f_j(x),j\in [m]$ are strongly convex with $\mu_j>0$, \eqref{assu:1} is satisfied with $\mu=\max\{\mu_j,j\in [m]\}$.
\end{remark}
The following lemma states any $y\in \mathcal{P}$ can be expressed as points of $\mathcal{V}$ and satisfies a useful proprety which is weaker than Lemma 5.3 of \cite{garber2016linearly}.
\begin{lemma}
	\label{lemma:3}
	Suppose $x$ can be written as some convex combination of $s$ number of vertices of $\mathcal{P}$: $x=\sum_{i=1}^s\lambda_i v_i$, where $\lambda_i\geq 0,\sum_{i=1}^s \lambda_i=1$ and $v_i\in \mathcal{V}$. Then any $y\in\mathcal{P}$ can be written as $y=\sum_{i=1}^{s}(\lambda_i-\Delta_i)v_i+\Delta z$, such that $z\in\mathcal{P}$, $\Delta_i\in[0,\lambda_i],\Delta=\sum_{i=1}^{s}\Delta_i$. And for any $i\in[s]$ such that $\Delta_i>0$, there exists at least an index $j_i\in[s_1]$ that satisfies $a_{j_i}^\top z=b_{j_i}$ and $a_{j_i}^\top v_i<b_{j_i}$.
\end{lemma}
\begin{proof}
	Consider the following optimization problem:
	\begin{align*}
		&\min \Delta=\sum_{i=1}^{s}\Delta_i\\
		s.t.&\  \frac{\Delta_i}{\lambda_i}= \delta\in[0,1],\forall i\in[s]\\
		&z=\frac{1}{\Delta}(y-\sum_{i=1}^{s}(\lambda_i-\Delta_i)v_i)\in\mathcal{P}\\
		&\max_{j\in [m]}\langle\nabla f_j(x),q-x\rangle\leq\max_{j\in [m]}\langle\nabla f_j(x),y-x\rangle,
	\end{align*}
	where $q=\sum_{i=1}^{s}(\lambda_i-\Delta_i)v_i+\sum_{i=1}^{s}p$. Since when $\Delta_i=\lambda,\forall i\in[s]$ we have $z=y$. It means that the feasible region of the problem is not empty and the optimal value of the problem is less than 1. Also $\Delta\geq 0$ means that the optimal value is nonnegative. There exists $i\in[s]$ such that $\Delta_i>0$, and an index $j_i\in[s_1]$ that satisfies $a_{j_i}^\top z=b_{j_i}$ and $a_{j_i}^\top v_i<b_{j_i}$. Otherwise we can find $\Delta_i'<\Delta_i$ and satisfies conditions.
\end{proof}
\begin{lemma}
	\label{lemma:4}
	Let $x\in \mathcal{P}$ and write $x$ as a convex combination of points in $\mathcal{P}$,i.e., $x=\sum_{i=1}^s\lambda_i v_i$ such that $\lambda_i>0$ for all $i=1,...,s$. Let $y\in \mathcal{P}$. $y$ can be written as a convex combination $y=\sum_{i=1}^s(\lambda_i -\Delta_i)v_i+\sum_{i=1}^s\Delta_i z$ for some $z\in \mathcal{P},\Delta_i\in [0,\lambda_i]$. Let $\Delta=\sum_{i=1}^s\Delta_i$ is minimized. We have
	\begin{equation}
		\sum_{i=1}^{s}\Delta_i\leq \frac{m_1\phi^{*}}{\xi^{*}}\Vert y-x\Vert.
	\end{equation}
\end{lemma}
\begin{proof}
	From Lemma \ref{lemma:3}, we know that exists $i \in[s]$ and $j_i$ such that $a_{j_i}^\top z=b_{j_i}$ and $a_{j_i}^\top v_i<b_{j_i}$. Let $C^*(z)\subset [m_1]$ be a set of minimal cardinality such that i) for all $j\in C^*(z)$, $a_j^\top z=b_j$, ii) for all $i\in [s]$ there exists $j_i\in C^*(z)$ for which $a_{j_i}^\top v_i<b_{j_i}$. Let $A_z\in R^{|C^*(z)|\times n}$ be the matrix $A$ that deletes each row $j\notin C^*(z)$. We have that
	\begin{eqnarray*}
		\Vert y-x\Vert^2\geq& \frac{1}{\Vert A_z\Vert^2}\Vert A_z(y-x)\Vert^2\\
		=&\frac{1}{\Vert A_z\Vert^2}\Vert A_z\sum_{i=1}^s \Delta_i(z-v_i)\Vert^2\\
		\geq& \frac{1}{\phi^{*2}}\sum_{j\in C^*(z)}(\sum_{i=1}^s \Delta_i(b_j-a_j^\top v_i))^2\\
		\geq& \frac{1}{\phi^{*2}\vert C^*(z)\vert}(\sum_{j\in C^*(z)}\sum_{i=1}^s \Delta_i(b_j-a_j^\top v_i))^2\\
		\geq& \frac{1}{\phi^{*2}\vert C^*(z)\vert}(\sum_{i=1}^s \Delta_i\xi^*)^2\\
		\geq& \frac{\xi^{*2}}{m_1\phi^{*2}}(\sum_{i=1}^s \Delta_i)^2.
	\end{eqnarray*}
	Hence we have
	\begin{equation*}
		\sum_{i=1}^{s}\Delta_i\leq \frac{m_1\phi^{*}}{\xi^{*}}\Vert y-x\Vert.
	\end{equation*}
\end{proof}
\begin{lemma}
	Let $\{x^k,k\in \mathbb{N}\}$ is generated by Algorithm \ref{Algo:1}. Let $(p^{PW}(x^k),q^{PW}(x^k))$ is the optimal solution of Problem (\ref{pro:1}) at $x^k$. Then for any $y\in \mathcal{P}$ we have
	\begin{eqnarray}
		\frac{\max_{j\in [m]}\langle \nabla f_j(x^k),y-x^k\rangle}{\Vert y-x^k\Vert}&\geq&\frac{m_1\phi^{*}}{\xi^{*}}\max_{j\in [m]}\langle \nabla f_j(x^k),p^{PW}(x^k)-q^{PW}(x^k)\rangle\nonumber\\
		&=&\frac{m_1\phi^{*}}{\xi^{*}}\theta^{PW}(x^k).
	\end{eqnarray}
	\label{lemma:5}
\end{lemma}
\begin{proof}
	$x^k$ can write the convex decomposition as $x^k=\sum_{i=1}^s \lambda_i v_i,\{v_i,i\in [s]\}\subset \mathcal{V}$. 
	Let $\Delta^k=\sum_{i=1}^s\Delta_i$. From Lemma \ref{lemma:4}, we know that $y=\sum_{i=1}^{s}(\lambda_i-\Delta_i) v_i+\Delta^k z$, where $z\in \mathcal{P}$, and $\sum_{i=1}^{s}\Delta_i\leq \frac{m_1\phi^{*}}{\xi^{*}}\Vert y-x^k\Vert$. Let $x^{\Delta^k}=\frac{1}{\Delta^k}\sum_{i=1}^{s}\Delta_i v_i\in\mathcal{P}$. Concider
	\begin{equation*}
		t^k=\sum_{i=1}^{s}(\lambda_i-\Delta_i) v_i+\Delta^k p(x^k,x^{\Delta^k}),
	\end{equation*}
	where $p(x^k,x^{\Delta^k})$ is the optimal solution of the following question:
	\begin{equation*}
		\min_{u\in \mathcal{P}}\max_{j\in [m]}\langle \nabla f_j(x^k),u-x^{\Delta^k}\rangle.
	\end{equation*}
	Since $p(x^k,x^{\Delta^k})-x^{\Delta^k}=\frac{1}{\Delta^k}(t^k-x^k)$ and $z-x^{\Delta^k}=\frac{1}{\Delta^k}(y-x^k)$, we have
	\begin{eqnarray*}
		\frac{1}{\Delta^k}\max_{j\in [m]}\langle \nabla f_j(x^k),t^k-x^k\rangle&=&\max_{j\in [m]}\langle \nabla f_j(x^k),p(x^k,x^{\Delta^k})-x^{\Delta^k}\rangle\\
		&\leq&\max_{j\in [m]}\langle \nabla f_j(x^k),z-x^{\Delta^k}\rangle\\
		&=&\frac{1}{\Delta^k}\max_{j\in [m]}\langle \nabla f_j(x^k),y-x^k\rangle.
	\end{eqnarray*}
	From $\Delta^k> 0$ we have that $\max_{j\in [m]}\langle \nabla f_j(x^k),t^k-x^k\rangle\geq\max_{j\in [m]}\langle \nabla f_j(x^k),y-x^k\rangle$. Now we can establish the relationship between $y$ and $(p^{PW}(x^k),q^{PW}(x^k))$.
	\begin{eqnarray*}
		\max_{j\in [m]}\langle \nabla f_j(x^k),y-x^k\rangle&\geq&\max_{j\in [m]}\langle \nabla f_j(x^k),t^k-x^k\rangle\\
		&=&\max_{j\in [m]}\langle \nabla f_j(x^k),\Delta^k p(x^k,x^{\Delta^k})-\sum_{i=1}^{s}\Delta_i v_i\rangle\\
		&=&\Delta^k\max_{j\in [m]}\langle \nabla f_j(x^k),p(x^k,x^{\Delta^k})-\frac{1}{\Delta^k}\sum_{i=1}^{s}\Delta_i v_i\rangle\\
		&\geq& \Delta^k\max_{j\in [m]}\langle \nabla f_j(x^k),p^{PW}(x^k)-q^{PW}(x^k)\rangle=\Delta^k\theta^{PW}(x^k).
	\end{eqnarray*}
	Combining the inequality and Lemma \ref{lemma:4}, we get the conclusion.
\end{proof}
\begin{remark}\normalfont
	\label{remark:3}
	Let $y=x^*$ which is a limit point of $\{x^k:k\in\mathbb{N}\}$ generated by Algorithm \ref{Algo:1}, and from Lemma \ref{lemma:5} we build the relation between $\theta^{PW}(x^k)$ and $\max_{j\in [m]}\langle \nabla f_j(x^k),x^*-x^k\rangle$. Acturally the coefficient $\frac{m_1\phi^{*}}{\xi^{*}}$ of $\theta^{PW}(x^k)$ is a constant with respect to $\mathcal{P}$. When $m=1$ this result degraded to the single-objective optimization. Then the facial distance in \cite{pena2019polytope} or the pyramidal width in \cite{lacoste2015global} can be applied to replace.
\end{remark}
\begin{theorem}
	Let $\{x^k:k\in \mathbb{N}\}$ be generated by Algorithm \ref{Algo:1} and $x^*$ is a limit point. Denote $\sigma(x^k,x^*) = \min_{j\in [m]}\{f_j(x^k)-f_j(x^*)\}$. Then we have
	\begin{equation}
		\sigma(x^k,x^*)=\mathcal{O}(\beta D^2/k), \forall k\geq 1.\label{eqtheo:1}
	\end{equation}
	Assume that $F(x)$ satisfies  \eqref{assu:1}, we have
	\begin{equation}
		\sigma(x^{k+1},x^*)\leq (1-\frac{2\mu\xi^{*2}}{2LD^2 m_1^2\phi^{*2}})\sigma(x^k,x^*).\label{eqtheo:2}
	\end{equation}
\end{theorem}
\begin{proof}
	Comining the proof of Theorem \ref{theo:1} and Remark \ref{remark:1}, \eqref{eqtheo:1} is obtained immediately. Now we prove \eqref{eqtheo:2} holds if \eqref{assu:1} is satisfied. \\
	First, we concider a good step. From \eqref{assu:1} of $F(x)$ and Lemma \ref{lemma:5} we have that
	\begin{eqnarray*}
		\theta^{PW}(x^k)&\leq&\frac{\xi^{*}}{m_1\phi^{*}\Vert x^*-x^k\Vert}\max_{j\in [m]}\langle \nabla f_j(x^k),x^*-x^k\rangle\\
		&\leq& -\frac{\sqrt{2\mu}\xi^{*}}{m_1\phi^{*}}\sqrt{\sigma(x^k,x^*)}.
	\end{eqnarray*}
	For a good step, the step size $\lambda_k=-\frac{\theta^{PW}(x^k)}{L\Vert d^{PW}(x^k)\Vert^2} <\lambda_{k,max}$. Accroding the smoothess of $F(x)$ we have that
	\begin{eqnarray*}
		f_j(x^{k+1})&\leq& f_j(x^k)-\lambda_k\langle \nabla f_j(x^k),d^{PW}(x^k)\rangle+\frac{L}{2}\lambda_k^2\Vert d^{PW}(x^k)\Vert^2\\
		&\leq& f_j(x^k)-\lambda_k\theta^{PW}(x^k)+\frac{L}{2}\lambda_k^2\Vert d^{PW}(x^k)\Vert^2.
	\end{eqnarray*}
	Hence we have that
	\begin{eqnarray*}
		\sigma(x^{k+1},x^*)&\leq& \sigma(x^k,x^*)-\lambda_k\theta^{PW}(x^k)+\frac{L}{2}\lambda_k^2\Vert d^{PW}(x^k)\Vert^2\\
		&=&\sigma(x^k,x^*)-\frac{\theta^{PW}(x^k)^2}{2L\Vert d^{PW}(x^k)\Vert^2}
		\leq \sigma(x^k,x^*)-\frac{2\mu\xi^{*2}}{2LD^2 m_1^2\phi^{*2}}\sigma(x^k,x^*).
	\end{eqnarray*}
	We get the conclusion.
\end{proof}
\section{Numberical Experiments}\label{sec:5}
This section we present some numberical experiments illustrating the good performance of DIPFWMOP. As a contrast, We consider Frank-Wolfe algorithm for multiobjective optimization\cite{assunccao2021conditional}, Away-step Frank-Wolfe algorithm for multiobjective optimization\cite{gonccalves2024away} and Projected Gradient method for multiobjective optimization\cite{drummond2004projected}. For simplicity, we call them as FW, AFW, DIPFW and PG in the follows. All of numberical experiments are implemented in PYTHON.

We have described the subproblem and the adaptive step size for FW, AFW and DIPFW in sect.\ref{sec:3}. It should be noted that adaptive step size may fail when accuracy is required. Hence the smoothness parameter each cycle is multiplied by a parameter $\alpha=0.5$. Now we will briefly describe PG. The subproblem of PG at each iteration is
\begin{eqnarray}
	\min_{u\in \mathcal{P}}\max_{j\in [m]}\langle \nabla f_j(x^k),u-x^k\rangle+\frac{1}{2}\Vert u-x^k\Vert^2.
\end{eqnarray}
It is a quadratic problem, so we use cvxpy to solve it. The stopping criterion we use is $\frac{1}{2}\Vert p(x^k)-x^k\Vert^2<\epsilon$ where $p(x^k)$ is the solution of subproblem of PG and $\epsilon>0$ is the tolerance. The step size we choose is the armijo line search.
We consider the following multiobjective optimization problem
\begin{eqnarray}
	&\min_{x\in \mathbb{R}^n}&F(x)=(\frac{1}{2}\Vert  y-b_1\Vert^2,...,\frac{1}{2}\Vert  y-b_m\Vert^2)^\top,\nonumber\\
	&s.t.&\ y=Ax,\\
	&&x\geq 0,e^\top x=1,\nonumber
\end{eqnarray}
where $e=(1,1,...,1)^\top\in R^{m}$, $A=(a_1,...,a_n)\in R^{p\times n}$, and $b_j\in R^p,j\in [m]$. $a_i$ of $A$ and $b_j,j\in [m]$ are selected ramdomly in $[0,1]^p$. We let $p=10,20,50$, $n=10$ and $m=2,3$. The tolerance $\epsilon$ is $10^{-4}$.

Tab.\ref{tab:1} lists the time and iteration result which four algorithms have run. From Fig.\ref{fig:1}, we know that DIPFW and PG is better than FW and AFW in most text problems. And DIPFW is faster than PG in time.

In Fig.\ref{fig:1}, we plot four algorithms' performance profile for time and iteration which is proposed in \cite{dolan2002benchmarking}. Fig.\ref{fig:1} similarly illustrates DIPFW is better than others.

\begin{table}
	\centering
	\vskip-0.1in
	\begin{tabular}{cccccccc}
		\hline
		&    &    &  time  &   &    &  iteration  &  \\
		\hline
		Method  &  Dim  &  Min  &  Mean  &  Max &  Min  &  Mean  &  Max\\
		\hline
		
		FW & (10,10,2) & 3.172 & 4.697 & 7.343 & 1457 & 2082.3 & 2923\\
		AFW & (10,10,2) & 3.234 & 4.784 & 6.672 & 1419 & 2074.9 & 2846\\
		DIPFW & (10,10,2) & 0.078 & 1.109 & 0.141 & 22 & 35.2 & 48\\
		PG & (10,10,2) & 0.156 & 0.220 & 0.281 & 21 & 27.8 & 37\\
		FW & (10,10,3) & 0.031 & 1.288 & 2.047 & 9 & 343 & 549\\
		AFW & (10,10,3) & 0.031 & 0.816 & 1.312 & 8 & 229.9 & 338\\
		DIPFW & (10,10,3) & 0.031 & 0.053 & 0.063 & 10 & 17.5 & 22\\
		PG & (10,10,3) & 0.094 & 0.159 & 0.281 & 13 & 19.1 & 32\\
		FW & (10,20,2) & 3.563 & 3.713 & 3.891 & 1715 & 1763.7 & 1797\\
		AFW & (10,20,2) & 0.078 & 0.097 & 0.125 & 38 & 44.8 & 59\\
		DIPFW & (10,20,2) & 0.047 & 0.061 & 0.078 & 17 & 22.6 & 27\\
		PG & (10,20,2) & 0.203 & 0.278 & 0.328 & 22 & 33.2 & 39\\
		FW & (10,20,3) & 0.156 & 0.530 & 1.344 & 60 & 235.2 & 596\\
		AFW & (10,20,3) & 0.125 & 0.541 & 1.375 & 59 & 230.3 & 582\\
		DIPFW & (10,20,3) & 0.047 & 0.053 & 0.063 & 15 & 18.8 & 23\\
		PG & (10,20,3) & 0.125 & 0.161 & 0.188 & 16 & 19.2 & 23\\
		FW & (10,50,2) & 0.469 & 3.958 & 16.063 & 208 & 1775.5 & 7228\\
		AFW & (10,50,2) & 0.5 & 4.233 & 18 & 207 & 1774.5 & 7227\\
		DIPFW & (10,50,2) & 0.094 & 0.114 & 0.141 & 27 & 33.9 & 41\\
		PG & (10,50,2) & 0.281 & 0.327 & 0.406 & 36 & 41.3 & 51\\
		FW & (10,50,3) & 0.219 & 0.653 & 0.969 & 96 & 274.7 & 418\\
		AFW & (10,50,3) & 0.234 & 0.673 & 1 & 95 & 273.7 & 417\\
		DIPFW & (10,50,3) & 0.063 & 0.066 & 0.781 & 21 & 24.2 & 28\\
		PG & (10,50,3) & 0.297 & 0.353 & 0.391 & 39 & 42.6 & 46\\
		
		\hline
	\end{tabular}
	\caption{time and iteration}
	\label{tab:1}
\end{table}

\begin{figure}[h]
	\centering
	\begin{subfigure}[b]{0.45\textwidth}
		\includegraphics[width=\textwidth]{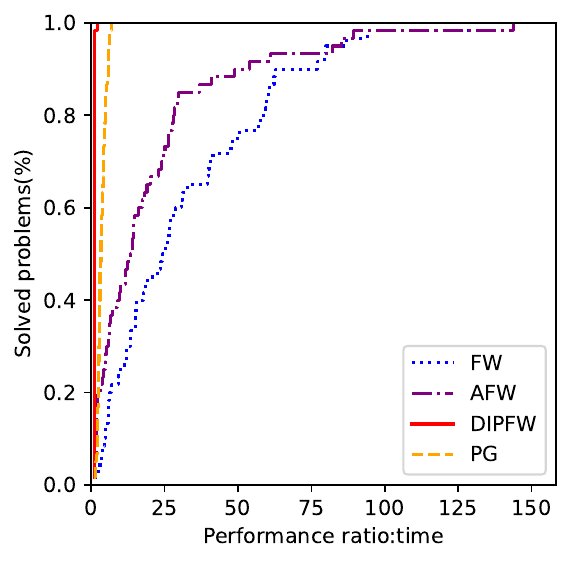}
		\caption{time}
		\label{}
	\end{subfigure}
	\hfill
	\begin{subfigure}[b]{0.45\textwidth}
		\includegraphics[width=\textwidth]{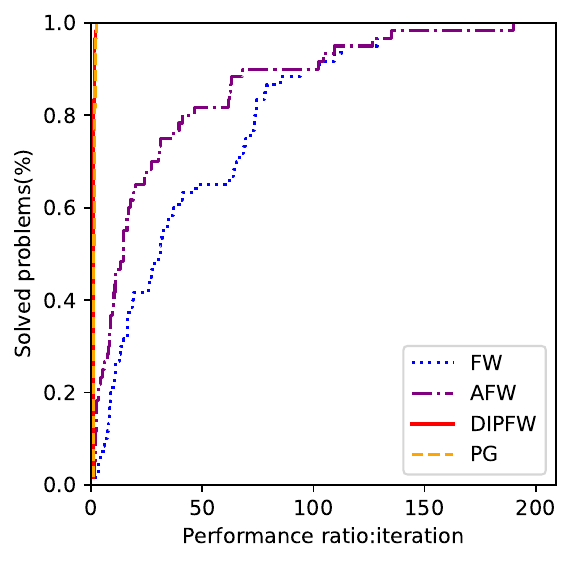}
		\caption{iteration}
		\label{}
	\end{subfigure}
	\caption{time(left) and iteration(right).}
	\label{fig:1}
\end{figure}

\section{Conclusion}\label{sec:6}
In this paper, we propose decomposition-invariant pairwise Frank-Wolfe algorithm for constrained multiobjective optimization. This algorithm solve constrained multiobjective optimization problems whose feasible region is an arbitrary bounded polytope. The descent direction of the algorithm combines pairwise direction. We establish asymptotic convergence analysis and linear convergence rate if objective functions are strongly convex.

In future work, we will extend this algorithm to the vector optimization and analysis its convergence.

\bibliographystyle{plain}
\bibliography{ref}

\end{document}